\documentclass[12pt]{amsart}
\usepackage{amscd,amssymb,amsmath,multicol}
\newtheorem{thm}[equation]{Theorem}
\numberwithin{equation}{section}

\newtheorem{diag}[equation]{Diagram}

\begin{document}
\raggedbottom \voffset=-.7truein \hoffset=0truein \vsize=8truein
\hsize=6truein \textheight=8truein \textwidth=6truein
\baselineskip=18truept
\def\cal{\mathcal}
\def\vareps{\varepsilon}
\def\mapright#1{\ \smash{\mathop{\longrightarrow}\limits^{#1}}\ }
\def\mapleft#1{\smash{\mathop{\longleftarrow}\limits^{#1}}}
\def\mapup#1{\Big\uparrow\rlap{$\vcenter {\hbox {$#1$}}$}}
\def\mapdown#1{\Big\downarrow\rlap{$\vcenter {\hbox {$\ssize{#1}$}}$}}
\def\on{\operatorname}
\def\span{\on{span}}
\def\cat{\on{cat}}
\def\a{\alpha}
\def\bz{{\Bbb Z}}
\def\gd{\on{gd}}
\def\imm{\on{imm}}
\def\sq{\on{Sq}}
\def\sspan{\on{span}^0}
\def\rank{\on{rank}}
\def\eps{\epsilon}
\def\br{{\Bbb R}}
\def\bc{{\Bbb C}}
\def\bh{{\Bbb H}}
\def\tfrac{\textstyle\frac}
\def\w{\wedge}
\def\b{\beta}
\def\A{{\cal A}}
\def\P{{\cal P}}
\def\zt{{\Bbb Z}_2}
\def\bq{{\Bbb Q}}
\def\ker{\on{ker}}
\def\coker{\on{coker}}
\def\u{{\cal U}}
\def\e{{\cal E}}
\def\exp{\on{exp}}
\def\wbar{{\overline w}}
\def\xbar{{\overline x}}
\def\ybar{{\overline y}}
\def\zbar{{\overline z}}
\def\ebar{{\overline e}}
\def\nbar{{\overline n}}
\def\mbar{{\overline m}}
\def\ubar{{\overline u}}
\def\ext{\on{Ext}}
\def\ni{\noindent}
\def\coef{\on{coef}}
\def\den{\on{den}}
\def\gd{{\on{gd}}}
\def\ev{{\text{ev}}}
\def\od{{\text{od}}}
\def\N{{\Bbb N}}
\def\Z{{\Bbb Z}}
\def\Q{{\Bbb Q}}
\def\R{{\Bbb R}}
\def\C{{\Bbb C}}
\def\notimm{\not\subseteq}
\def\tmf{\on{tmf}}
\title[New nonimmersion results]
{Some new nonimmersion results for real projective spaces}

\author{Donald M. Davis}
\address{Department of Mathematics, Lehigh University\\Bethlehem, PA 18015, USA}
\email{dmd1@lehigh.edu}
\date{January 28, 2011}


\maketitle
\begin{abstract} We use the spectrum tmf to obtain new nonimmersion results
for many real projective spaces $RP^n$ for $n$ as small as 113. The only new ingredient
is some new calculations of tmf-cohomology groups. We present an expanded table of nonimmersion results.
Our new theorem is new for 17\% of the values of $n$ between $2^i$ and $2^i+2^{14}$ for $i\ge15$.
 \end{abstract}
\section{Introduction}\label{intro} We use the spectrum $\tmf$ to prove the following new nonimmersion theorem
for real projective spaces $P^n$.
\begin{thm}\label{mainthm} Let $\a(n)$ denote the number of 1's in the binary expansion of $n$.
\begin{enumerate}
\item[a.] If $\a(M)=3$, then $P^{8M+9}$ does not immerse in ($\not\subseteq$) $\br^{16M-1}$.
\item[b.] If $\a(M)=6$, then $P^{8M+9}\notimm\br^{16M-11}$.
\item[c.] If $\a(M)=7$, then $P^{16M+16}\notimm\br^{32M-7}$ and $P^{16M+17}\notimm\br^{32M-6}$.
\item[d.] If $\a(M)=9$, then $P^{32M+25}\notimm\br^{64M-4}$ and $P^{32M+26}\notimm\br^{64M-3}$.
\item[e.] If $\a(M)=10$, then $P^{16M+17}\notimm\br^{32M-20}$ and $P^{16M+18}\notimm\br^{32M-19}$.
\end{enumerate}
\end{thm}
We apply the same method that was used in \cite{BDM}, using $\tmf^*(-)$ to detect nonexistence of axial maps.
The novelty here is that we compute and utilize groups $\tmf^*(P^m\w P^n)$ when $m$ and/or $n$ is odd.
In \cite{BDM}, only even values of $m$ and $n$ were studied. There is, however, no significant difference
or complication in using the odd values. We prove Theorem \ref{mainthm} in Section \ref{pfsec}.

For many years, the author has maintained a website (\cite{tbl}) which listed all known immersion, nonimmersion,
embedding, and nonembedding results for $P^n$ and tabulated them for $n=2^i+d$ with $2^i>d$ and $0\le d\le 63$.
In \cite{KW}, W.~Stephen Wilson acknowledged how this table motivated him to try (and succeed) to prove nonimmersions
for small $P^n$. Our Theorem \ref{mainthm}(a.) includes $P^{2^i+49}\notimm\br^{2^{i+1}+79}$ and $P^{2^i+57}\notimm
\br^{2^{i+1}+95}$ for $i\ge6$, which improve on previous best results (of \cite{KW}) by 1 and 2 dimensions, respectively,
and hence enter the table \cite{tbl}.

To facilitate checking whether results are new, the author has greatly expanded his table of nonimmersion results
at {\tt www.lehigh.edu/}$\sim${\tt dmd1/imms.html}. We have listed there the best known nonimmersions for $P^{2^i+d}$
for $2^i>d$ and $0\le d\le 16,\!383$ together with the first acknowledged source. A listing of and link to the {\tt Maple} program
that generated this table is also included there. This table gives all known nonimmersion results for $P^n$ with
$7<n<49,\!152$ except for James' nonimmersions of $P^{2^e-1}$ in dimension $2^{e+1}-2e-\langle 3,2,2,4\rangle$ if $e\equiv \langle0,1,2,3\rangle$ mod 4.(\cite{J})

Theorem \ref{mainthm} appears 2796 times in this table, thus giving new results for 17\% of the projective spaces of dimension
between $2^{i}$ and $2^{i}+2^{14}$ for $i\ge15$. The seminal result of \cite{Ann},
\begin{equation}\label{Anneq}P^{2(m+\a(m)-1)}\notimm\br^{4m-2\a(m)},\end{equation}
appears 7063 times in the table, but is divided among four references. The first 4361 of them appeared in \cite{AD}, which obtained
a result equivalent to (\ref{Anneq}) for $P^n$ with $n$ satisfying a very complicated condition. The statement (\ref{Anneq})
was first conjectured in \cite{BD} and proved there for $\a(m)\le6$, which yielded 168 new results in this table. It was
extended to $\a(m)=7$ in \cite{W}, and this still applies to 700 values. This left 1834 values which were covered by the general result (\ref{Anneq})
and not by any of the three preceding references, and have not been bettered in subsequent work.

The first tmf-paper, (\cite{BDM}), appears 2866 times in the table; there are 110 additional values for small $\alpha(-)$ of tmf-implied
nonimmersions which were overlooked
in \cite{BDM} and noted in \cite{DM}. The other big collection of nonimmersion results is those obtained in \cite{KW} using $ER(2)$-cohomology,
which appears 2092 times.
Both $ER(2)$ and tmf can be considered as real-versions of $BP\langle2\rangle$. Using $ER(2)$ is advantageous because $ER(2)^*(P^n)$
has a 2-dimensional class, while $\tmf^*(P^n)$ only has an 8-dimensional class. Also $ER(2)$ is more closely related to $BP\langle2\rangle$,
and so, as W.~Stephen Wilson says, it can ``mooch" off the result (\ref{Anneq}). The advantage of tmf is that some of its groups are one 2-power larger
than those of $ER(2)$.

In \cite{Ann}, it was stated that (\ref{Anneq}) was within 2 dimensions of all known nonimmersion results, in the sense that
the two dimensions could come from the Euclidean space, the projective space, or a combination. In other words, if $D(n)$
denotes the nonimmersion dimension for $P^n$ obtained from (\ref{Anneq}), and $K(n)$ the best known nonimmersion dimension for $P^n$,
then, at the time, it was true that
\begin{equation}\label{KD2}K(n)\le \max(D(n)+2,\ D(n+1)+1,\ D(n+2)).\end{equation}
This is no longer true. There are 10 values of $n$ in the table for which the result of \cite{DZ}, which states that if $\a(n)=4$ and $n\equiv10$ mod 16
then $P^n\notimm\br^{2n-9}$, does not satisfy (\ref{KD2}), and there are 418 values of $n$ in the table for which Theorem 1.1(c) does not
satisfy (\ref{KD2}). These are the only results which are more than 2 stronger than (\ref{Anneq}) in the sense of (\ref{KD2}), and it is still true
that (\ref{Anneq}) is within 3 dimensions of all known results in the same sense. That is, the following statement is currently true.
$$K(n)\le\max(D(n)+3,\ D(n+1)+2,\ D(n+2)+1,\ D(n+3)).$$
The first example of (\ref{KD2}) not being satisfied occurs for $n=58$; we have $K(58)=107$ due to \cite{DZ} (which used modified Postnikov towers)
while $D(58)=D(59)=98$ and $D(60)=D(61)=106$. The first example of our \ref{mainthm}(c) causing (\ref{KD2}) to be not satisfied occurs from
$K(3584)=7129$ (due to \ref{mainthm}(c)) while $D(3584)=D(3585)=7124$ and $D(3586)=D(3587)=7128$.

Theorem \ref{mainthm} can be extended to larger values of $\a(M)$ similarly to what was done in \cite{BDM}.
We have emphasized the results for small values of $\a(M)$ for clarity of exposition. The extension, whose proof
we sketch in Section \ref{bigMsec}, is as follows. The lettering of the parts corresponds to the parts of Theorem
\ref{mainthm}.
\begin{thm}\label{bigMthm} Let $p(h)$ denote the smallest 2-power $\ge h$.
\begin{enumerate}
\item[b,e.] Suppose $\a(M)=4h+2$ and $h\le 2^{e_1}-2^{e_0}$ if $M\equiv 2^{e_0}+2^{e_1}\mod 2^{e_1+1}$ with $e_0<e_1$. Then
\begin{enumerate}
\item[b.] If $h$ is odd, $P^{8M+8h+1}\notimm\br^{16M-8h-3}$, and
\item[e.] If $h$ is even, then $P^{8M+8h+1}\notimm\br^{16M-8h-4}$ and $P^{8M+8h+2}\notimm\br^{16M-8h-3}$.
\end{enumerate}
\item[c.] If $\a(M)=4h+3$ with $h$ odd and $M\equiv0\mod p(h+1)$, then $P^{8M+8h+8}\notimm\br^{16M-8h+1}$ and
$P^{8M+8h+9}\notimm\br^{16M-8h+2}$.
\item[d.] If $\a(M)=4h+1$ with $h$ even and $M\equiv0\mod p(h+1)$, then $P^{8M+8h+9}\notimm\br^{16M-8h+12}$ and
$P^{8M+8h+10}\notimm\br^{16M-8h+13}$.
\end{enumerate}
\end{thm}

\section{Proof of Theorem \ref{mainthm}}\label{pfsec}
Let tmf denote the 2-local connective spectrum introduced in \cite{HM}, whose mod-2 cohomology is the quotient of the mod-2
Steenrod algebra $A$ by the left ideal generated by $\sq^1$, $\sq^2$, and $\sq^4$. Thus $\tmf_*(X)$ may be computed by the Adams spectral
sequence (ASS) with $E_2=\ext_{A_2}(H^*X,\zt)$, where $A_2$ is the subalgebra of $A$ generated by $\sq^1$, $\sq^2$, and $\sq^4$.
We rely on Bob Bruner's software (\cite{Bru}) for our calculations of these Ext groups.
It was proved in \cite[p.167]{DM} that there are 8-dimensional classes $X$, $X_1$, and $X_2$ such that the
homomorphism in $\tmf^*(-)$ induced by an axial map $P^m\times P^n\to P^k$ effectively sends $X$ to $u(X_1+X_2)$, where $u$ is a unit
in $\tmf^0(P^m\times P^n)$ which will be omitted from our exposition.

We will often use duality isomorphisms $\tmf^i(P^n)\approx \tmf_{-i-1}(P_{-n-1})$ for $i>2$, and $\tmf^i(P^m\w P^n)\approx
\tmf_{-i-2}(P_{-m-1}\w P_{-n-1})$ for $i>\max(m,n)+ 2$. For any integer $m$, $P_m$ denotes the spectrum
$P_m^\infty$. We make frequent use of the periodicity $P_{b+8}^{t+8}\w\tmf\simeq \Sigma^8P_b^t\w\tmf$ proved in \cite[Prop 2.6]{BDM}.

We let $\nu(-)$ denote the exponent of 2 in an integer, and use $\nu\bigl(\binom mn\bigr)=\a(n)+\a(m-n)-\a(n)$.
Also, if $L$ is large, $\nu\bigl(\binom{2^L-k}n\bigr)=\nu\bigl(\binom{-k}n\bigr)=\nu\bigl(\binom{n+k-1}n\bigr)$. We will never be interested in the values
of odd factors of coefficients, and will not list them.

\begin{proof}[Proof of (a).] If the immersion exists, there is an axial map $P^{8M+9}\times P^{8M+9}\to P^{16M-1}$. The induced homomorphism
in $\tmf^*(-)$ sends $0=X^{2M}$ to \begin{equation}\label{sum1}\sum\tbinom{2M}iX_1^iX_2^{2M-i}\end{equation} in $\tmf^{16M}(P^{8M+9}\w P^{8M+9})$. This group is isomorphic to
$\tmf_{-2}(P_{-10}\w P_{-10})\approx\tmf_{30}(P_6\w P_6)$. The portion of the ASS for $\tmf_{30}(P_6\w P_6)$ arising from filtration 0
by $h_0$-extensions appears in Diagram \ref{diaga}.

\begin{center}
\begin{minipage}{6.5in}
\begin{diag}\label{diaga}{Portion of $\tmf_{30}(P_6\w P_6)$}
\begin{center}
\begin{picture}(30,50)
\def\mp{\multiput}
\def\elt{\circle*{3}}
\put(0,0){\line(1,0){30}}
\mp(15,0)(0,15){4}{\elt}
\put(15,0){\line(0,1){45}}
\mp(11,0)(8,0){2}{\line(0,1){15}}
\mp(11,0)(0,15){2}{\elt}
\mp(19,0)(0,15){2}{\elt}
\end{picture}
\end{center}
\end{diag}
\end{minipage}
\end{center}

\noindent There are several elements in higher filtration which are not relevant to our argument. The elements pictured in
Diagram \ref{diaga} cannot be hit by differentials in the ASS because in dimension 31 there is only one tower in low enough
filtration and it cannot support a differential by the argument of \cite[p.54]{BDM}, namely that its generator is a constructible
homotopy class. The filtration-0 elements must correspond to $X_1^{M-1}X_2^{M+1}$, $X_1^MX_2^M$, and $X_1^{M+1}X_2^{M-1}$
in $\tmf^{16M}(P^{8M+9}\w P^{8M+9})$. Since
\begin{equation}\label{one}2^2X^{M+1}=0\text{ in }\tmf^*(P^{8M+9}),\end{equation}
 the two $\bz/4$'s in Diagram \ref{diaga} must
represent $X_1^{M\pm1}X_2^{M\mp1}$, and multiples of these are 0 in all filtrations $>1$.
 Thus  $X_1^MX_2^M$ generates the $\bz/2^4$ in $\tmf^{16M}(P^{8M+9}\w P^{8M+9})$. Since $\nu\bigl(\binom {2M}M\bigr)=\a(M)=3$, we obtain that (\ref{sum1}) is nonzero,
contradicting the existence of the immersion.
\end{proof}

\begin{proof}[Proof of (b).] If the immersion exists, there is an axial map $P^{8M+9}\times P^{2^{L+3}-16M+9}\to P^{2^{L+3}-8M-11}$
for sufficiently large $L$. Hence
\begin{equation}\label{sum2}\sum\tbinom{-M-1}iX_1^iX_2^{2^L-M-1-i}=0\in\tmf^{2^{L+3}-8M-8}(P^{8M+9}\w P^{2^{L+3}-16M+9}).\end{equation}
This group is isomorphic to $\tmf_{38}(P_6\w P_6)$, and the relevant part of it is given in Diagram \ref{diagb}. Similarly to case (a),
and continuing in all remaining cases, it cannot be hit by a differential in the ASS.

\begin{center}
\begin{minipage}{6.5in}
\begin{diag}\label{diagb}{Portion of $\tmf_{38}(P_6\w P_6)$}
\begin{center}
\begin{picture}(30,80)
\def\mp{\multiput}
\def\elt{\circle*{3}}
\put(0,0){\line(1,0){30}}
\mp(12,0)(0,15){4}{\elt}
\mp(12,0)(6,0){2}{\line(0,1){45}}
\mp(7,0)(16,0){2}{\line(0,1){15}}
\mp(7,0)(0,15){2}{\elt}
\mp(23,0)(0,15){2}{\elt}
\put(12,45){\line(1,5){3}}
\put(18,45){\line(-1,5){3}}
\mp(18,0)(0,15){4}{\elt}
\mp(15,60)(0,15){2}{\elt}
\put(15,60){\line(0,1){15}}
\end{picture}
\end{center}
\end{diag}
\end{minipage}
\end{center}

The outer ($\bz/4$) generators must correspond to $X_1^{M-2}X_2^{2^L-2M+1}$ and $X_1^{M+1}X_2^{2^L-2M-2}$. (Note that 4 times
each of these classes is 0 by (\ref{one}), and so they cannot produce a higher-filtration component impacting the middle summands.
This will be the case also for the outer summands in subsequent diagrams.) The inner generators must
be $X_1^{M-1}X_2^{2^L-2M}$ and $X_1^{M}X_2^{2^L-2M-1}$. By \cite[Thm 2.7]{BDM}, the class $2^4(X_1^{M-1}X_2^{2^L-2M}+X_1^{M}X_2^{2^L-2M-1})$
is 0 in filtration 4, although it might be nonzero in filtration 5. This is depicted by the behavior of the chart
between filtration 3 and 4. Since $\a(M)=6$, the component of these terms in (\ref{sum2})
is $$\tbinom{-M-1}{M-1}X_1^{M-1}X_2^{2^L-2M}+\tbinom{-M-1}MX_1^{M}X_2^{2^L-2M-1}=2^5X_1^{M-1}X_2^{2^L-2M}+2^6X_1^{M}X_2^{2^L-2M-1},$$ which is nonzero in the group depicted by Diagram \ref{diagb}, contradicting the existence of the immersion.
\end{proof}

\begin{proof}[Proof of (c).] If the first immersion exists, there is an axial map $P^{16M+16}\times P^{2^{L+3}-32M+5}\to P^{2^{L+3}-16M-18}$. Hence
\begin{equation}\label{sum3}\sum\tbinom{-2M-2}iX_1^iX_2^{2^L-2M-2-i}=0\in\tmf^*(P^{16M+16}\w P^{2^{L+3}-32M+5}).\end{equation}
This group is isomorphic to $\tmf_{46}(P_7\w P_2)$, and the relevant part of it is given in the left side of Diagram \ref{diagc}.

\begin{center}
\begin{minipage}{6.5in}
\begin{diag}\label{diagc}{Portion of $\tmf_{46}(P_7\w P_2)$ and $\tmf_{46}(P_6\w P_3)$}
\begin{center}
\begin{picture}(160,110)
\def\mp{\multiput}
\def\elt{\circle*{3}}
\put(0,0){\line(1,0){35}}
\mp(12,0)(0,15){4}{\elt}
\mp(12,0)(6,0){2}{\line(0,1){45}}
\mp(7,0)(17,0){2}{\line(0,1){45}}
\mp(7,0)(0,15){4}{\elt}
\mp(24,0)(0,15){4}{\elt}
\mp(18,0)(0,15){4}{\elt}
\put(29,0){\elt}
\mp(12,45)(6,0){2}{\line(1,5){3}}
\mp(18,45)(6,0){2}{\line(-1,5){3}}
\put(21,60){\elt}
\mp(15,60)(0,15){4}{\elt}
\put(15,60){\line(0,1){45}}
\put(100,0){\line(1,0){35}}
\mp(112,0)(0,15){4}{\elt}
\mp(112,0)(6,0){2}{\line(0,1){45}}
\mp(107,0)(17,0){2}{\line(0,1){45}}
\mp(107,0)(0,15){4}{\elt}
\mp(124,0)(0,15){4}{\elt}
\mp(118,0)(0,15){4}{\elt}
\mp(129,0)(0,15){2}{\elt}
\put(129,0){\line(0,1){15}}
\mp(112,45)(6,0){2}{\line(1,5){3}}
\mp(118,45)(6,0){2}{\line(-1,5){3}}
\mp(121,60)(0,15){2}{\elt}
\put(121,60){\line(0,1){15}}
\mp(115,60)(0,15){4}{\elt}
\put(115,60){\line(0,1){45}}
\end{picture}
\end{center}
\end{diag}
\end{minipage}
\end{center}

The generators, from left to right, correspond to $X_1^{2M-2}X_2^{2^L-4M},\ldots,X_1^{2M+2}X_2^{2^L-4M-4}$, with the sum relation in
filtration 4 similar to that of the previous (and future) parts. Since $\a(M)=7$, the component of the middle terms in (\ref{sum3})
is $$2^{8+\nu(M)}X_1^{2M-1}X_2^{2^L-4M-1}+2^7X_1^{2M}X_2^{2^L-4M-2}+2^8X_1^{2M+1}X_2^{2^L-4M-3},$$ which is nonzero in
the group depicted by Diagram \ref{diagc}. The argument for the second nonimmersion involves the same sum in a  group
isomorphic to $\tmf_{46}(P_6\w P_3)$, which is pictured on the right side of Diagram \ref{diagc}.
\end{proof}

\begin{proof}[Proof of (d).] The proof is similar to those of parts (b) and (c). The first nonimmersion is proved by showing if $\a(M)=9$, then
\begin{equation}\label{sum4} \sum\tbinom{-4M-3}iX_1^iX_2^{2^L-4M-3-i}\ne0\in\tmf^{2^{L+3}-32M-24}(P^{32M+25}\w P^{2^{L+3}-64M+2}).\end{equation}
This group is isomorphic to $\tmf_{62}(P_6\w P_5)$, the relevant part of which is depicted in Diagram \ref{diagd},
with generators corresponding to $i=4M-3,\ldots,4M+3$ in (\ref{sum4}). The sum relation in filtration 8 follows
from \cite[Thm 2.7]{BDM}. The middle components of our class are
$$2^{10+\nu(M)}X_1^{4M-1}X_2^{2^L-8M-2}+2^9X_1^{4M}X_2^{2^L-8M-3}+2^9X_1^{4M+1}X_2^{2^L-8M-4},$$
which is nonzero in filtration 9. Note that $2^9X_1^{4M}X_2^{2^L-8M-3}$ is 0 in filtration 9, as can be seen from Diagram \ref{diagd} or from \cite[2.7]{BDM}, which says that if $g_1$, $g_2$, $g_3$ denote the middle three generators, then
there are relations that both $2^8(g_1+g_2+g_3)$ and $2^8(g_1+g_3)$ have filtration $>8$.

\begin{center}
\begin{minipage}{6.5in}
\begin{diag}\label{diagd}{Portion of $\tmf_{62}(P_6\w P_5)$}
\begin{center}
\begin{picture}(50,140)
\def\mp{\multiput}
\def\elt{\circle*{3}}
\put(0,0){\line(1,0){50}}
\mp(12,0)(0,15){4}{\elt}
\mp(12,0)(6,0){5}{\line(0,1){45}}
\put(7,0){\line(0,1){30}}
\mp(7,0)(0,15){3}{\elt}
\mp(24,0)(0,15){4}{\elt}
\mp(18,0)(0,15){4}{\elt}
\mp(30,0)(0,15){4}{\elt}
\mp(36,0)(0,15){4}{\elt}
\mp(41,0)(0,15){2}{\elt}
\mp(12,45)(6,0){4}{\line(1,5){3}}
\mp(18,45)(6,0){4}{\line(-1,5){3}}
\mp(21,60)(0,15){4}{\elt}
\mp(27,60)(0,15){4}{\elt}
\mp(33,60)(0,15){2}{\elt}
\mp(15,60)(0,15){3}{\elt}
\put(15,60){\line(0,1){30}}
\mp(21,60)(6,0){2}{\line(0,1){45}}
\put(33,60){\line(0,1){15}}
\put(21,105){\line(1,5){3}}
\put(27,105){\line(-1,5){3}}
\put(24,120){\line(0,1){15}}
\mp(24,120)(0,15){2}{\elt}
\put(41,0){\line(0,1){15}}
\end{picture}
\end{center}
\end{diag}
\end{minipage}
\end{center}

The argument for the second nonimmersion is virtually identical. Its obstruction is the same sum in a group
isomorphic to $\tmf_{62}(P_5\w P_6)$, so just the reverse of Diagram \ref{diagd}.
\end{proof}

\begin{proof}[Proof of (e).] The obstruction this time is $\sum\binom{-2M-2}iX_1^iX_2^{2^L-2M-2-i}$ in a group
isomorphic to the one depicted in Diagram \ref{diagd}. The middle terms are
$$2^9X_1^{2M-2}X_2^{2^L-4M}+2^{11+\nu(M)}X_1^{2M-1}X_2^{2^L-4M-1}+2^{10}X_1^{2M}X_2^{2^L-4M-2},$$
which is nonzero.\end{proof}

\section{Sketch of proof of Theorem \ref{bigMthm}}\label{bigMsec}
We use the $v_2^8$-periodicity of $\ext_{A_2}$ proved in \cite[p.299,Thm 5.9]{A2} to see that,
if one of the diagrams
of Section \ref{pfsec} depicts a portion of $\tmf_i(P_a\w P_b)$, then the top part of the portion of $\tmf_{i+48j}(P_a\w P_b)$
generated by filtration-0 classes has the same form $8j$ units higher. We also use the arguments on \cite[p.54]{BDM}
to see that, when this portion is interpreted as a quotient of a $\tmf^k(P^c\w P^d)$ group, the relations are of the same sort
as those in \cite[Thm 2.7]{BDM}. The relation \cite[(2.10)]{BDM} is especially important and will be noted specifically below.
We use cofiber sequences such as $S^a\w P_b\to P_a\w P_b\to P_{a+1}\w P_b$ to deduce results for our spaces, in which at least one
of the bottom dimensions is even, from those of \cite{BDM}, which dealt with the situation when both bottom dimensions are odd.
The nice form of $\ext_{A_2}(H^*P_b)$ below a certain line of slope 1/6 is important here. As noted on
\cite[p.54]{BDM}, it is just a sum of copies of $\ext_{A_1}(\zt)$, suitably placed.

\begin{proof}[Proof of \ref{bigMthm}(b,e).] If the immersion in (b) exists, there is an axial map
$$P^{8M+8h+1}\times P^{2^{L+3}-16M+8h+1}\to P^{2^{L+3}-8M-8h-3}.$$
We obtain a contradiction to this by showing
\begin{equation}\sum\tbinom{-M-h}iX_1^iX_2^{2^L-M-h-i}\ne0\in\tmf^*(P^{8M+8h+1}\w P^{2^{L+3}-16M+8h+1}).\label{sum5}\end{equation}
Our obstruction will be in filtration $4h+1$, where there is a nonzero class by $v_2^8$-periodicity from Diagram \ref{diagb},
which is the case $h=1$. Note that the group in which (\ref{sum5}) lies is isomorphic to $\tmf_{24h+14}(P_6\w P_6)$.
The terms in (\ref{sum5}) with $i>M$ cannot interfere in this filtration because for such $i$, $2^{4h-2}X_1^i=0$ in
$\tmf^*(P^{8M+8h+1})$. The same holds for terms with $i<M-h$ due to the second factor. By \cite[3.12]{BDM}, the coefficients
of the terms in (\ref{sum5}) with $M-h\le i\le M$ are all divisible by $2^{\a(M)-1}=2^{4h+1}$. This is where the strange hypothesis
comes into play. Next we note that
$$\nu\bigl(\sum_{j=0}^h\tbinom h{j}\tbinom{-M-h}{M-j}\bigr)=\nu\bigl(\tbinom{-M}M\bigr)=\a(M)-1.$$
By a variant on \cite[Cor 2.13.3]{BDM}, this implies that (\ref{sum5}) is nonzero. There are four things that are required to make this
work. (a) No interference from the outer terms because they are precisely 0 in a lower filtration. (b) All the $h+1$ intermediate
terms have filtration at least $4h+1$. (c) The chart is nonzero in filtration $4h+1$. (d) An odd number of the intermediate terms
which have $\binom hj$ odd, $0\le j\le h$, are nonzero in filtration $4h+1$. This latter is a version of \cite[(2.10)]{BDM}. It is a consequence of
a relation in every fourth filtration that the sum of the basic classes in the previous filtration is 0 in that filtration.
By ``basic,'' we mean those obtained from canonical classes in filtration 0 or 4 by $v_2^8$ periodicity.

The proof of (e) is virtually identical.\end{proof}

\begin{proof}[Proof of \ref{bigMthm}(c,d).] The proof of (d) is virtually identical to that of (c), and this is similar to that
of (b) with the main difference being that the obstruction is due to $\binom{-M-1}M$ instead of $\binom{-M}M$, which causes a
very different-looking hypothesis. The contradiction to the first result of (c) is obtained by showing
\begin{equation}\label{sum6}\sum\tbinom{-M-h-1}iX_1^iX_2^{2^L-M-h-1-i}\ne0\in\tmf^*(P^{8M+8h+8}\w P^{2^{L+3}-16M+8h-3}).\end{equation}
The obstruction will be in filtration $\a(M)=4h+3$. The terms with $i>M$ or $i<M-h$ are precisely 0 in  filtration less than $4h+3$ due to
their first or second factor. By our hypothesis and \cite[3.8]{BDM}, the intermediate terms are all divisible by $2^{\a(M)}$.
Since $$\nu\bigl(\sum_{j=0}^h\tbinom h{j}\tbinom{-M-h-1}{M-j}\bigr)=\nu\bigl(\tbinom{-M-1}M\bigr)=\a(M),$$
and, by $v_2^8$-periodicity from Diagram \ref{diagc}, the obstruction group is nonzero in filtration $\a(M)=4h+3$.
\end{proof}

\def\line{\rule{.6in}{.6pt}}

\end{document}